\documentclass[11pt,a4paper]{article}

\usepackage[T1]{fontenc}
\usepackage[utf8]{inputenc}
\usepackage{lmodern}
\usepackage[breaklinks]{hyperref}
\usepackage{graphicx}
\usepackage{lipsum}
\usepackage{amsmath}
\usepackage{mathtools}
\usepackage{float}
\usepackage{textcomp}
\usepackage{amscd}
\usepackage{pgfplots}
\usepackage{amsfonts}
\usepackage{amssymb}
\usepackage{amsthm}
\usepackage{exscale}
\usepackage{mathrsfs}
\usepackage{slashed}
\usepackage{stmaryrd}
\usepackage{tikz}
\usetikzlibrary{
    calc,
    intersections,
    angles,
    quotes,
    arrows.meta,
    decorations.markings
}
\usepackage{tikz-cd}
\usepackage{enumitem}
\usepackage{accents}
\theoremstyle{definition}
\newtheorem{thm}{Theorem}[section]

\newtheorem{prop}[thm]{Proposition}

\theoremstyle{remark}
\newtheorem{rem}[thm]{Remark}

\numberwithin{equation}{section}

\usepackage{accents}
\newcommand{\MG}{\underaccent{\longsim}{\mathcal{M}}} 
\newcommand{\longsim}{\mathrel{\scalebox{1.3}[0.8]{\ensuremath{\sim}}}}

\newcommand{\pic}[1]{\textrm{Pic}\left(#1\right)}

\def\XXint#1#2#3{{\setbox0=\hbox{$#1{#2#3}{\int}$}
\vcenter{\hbox{$#2#3$}}\kern-.5\wd0}}

\newcommand{\RightAngle}[5]{%
  \coordinate (RAa) at ($(#1)!#4!(#2)$);
  \coordinate (RAb) at ($(#1)!#5!(#3)$);
  \coordinate (RAc) at ($(RAa)+(RAb)-(#1)$);
  \draw (RAa)--(RAc)--(RAb);
}
\newcommand{\QuoteAngle}[6]{%
  \draw
    pic[
      draw,
      #6,
      angle radius=#4cm,
      angle eccentricity=1.35,
      "#5"
    ]
    {angle=#2--#1--#3};
}
\newcommand{\PerpMark}[5]{%
  \coordinate (PM) at ($(#1)!#3!(#2)$);
  \draw[#5]
    ($(PM)!#4!90:(#2)$) --
    ($(PM)!#4!-90:(#2)$);
}

\usepackage[framemethod=TikZ]{mdframed}
\newmdenv[
  backgroundcolor=gray!10,
  linecolor=black,
  linewidth=1pt,
  roundcorner=5pt,
  skipabove=10pt,
  skipbelow=10pt
]{boxedquote}
\usepackage{pdflscape}
\usepackage{adjustbox}
\usepackage[most]{tcolorbox}
\newtcolorbox{summarybox}{
  title=Concetti chiave,
  colback=gray!10,
  colframe=black,
  fonttitle=\bfseries
}

\title{On the Reconstruction of SAS from Other Triangle Congruence Criteria}
\author{Roberto Volpe}
\date{September 11, 2026}

\begin{document}

\maketitle

\begin{abstract}
Starting from a Hilbert plane and removing the Side-Angle-Side (SAS) congruence axiom, we
investigate to what extent SAS can be recovered synthetically from the remaining classical
triangle congruence criteria. We show that the Angle-Side-Angle criterion, together with a
ray correspondence principle corresponding to Theorem 13 of Hilbert's \emph{Grundlagen der
Geometrie}, suffices to reconstruct SAS. We further show that both the Side-Side-Side and
the Side-Angle-Angle criteria also suffice, once combined with the ray correspondence
principle and suitable auxiliary principles -- the existence of midpoints and a
hypotenuse-angle criterion for right triangles in the first case, and the existence of angle
bisectors, the congruence of supplements of congruent angles, and the Pons Asinorum in the
second. Although the two routes rely on auxiliary principles of different character, we show
that they converge on a single final argument once a common hypotenuse-angle criterion is
established. A metamathematical analysis, based on an explicit model adapted from Hilbert's
own independence construction, complements these reconstructions: it shows that the ray
correspondence principle alone cannot reconstruct any of the classical criteria, and that the
Pons Asinorum and the hypotenuse-angle criterion are each independent of the remaining
auxiliary principles used in their respective reconstructions. The resulting picture is not a
formal hierarchy of the congruence criteria, but it does show that the Angle-Side-Angle
reconstruction rests on a provably more economical basis than those obtained from
Side-Side-Side or Side-Angle-Angle.
\end{abstract}

\section{Introduction}\label{PEN_s1}

The axiomatic foundations of Euclidean geometry are closely connected with the problem of determining which geometric principles are genuinely independent and which can be derived from more elementary assumptions. Hilbert's \emph{Foundations of Geometry} provided a systematic axiomatic framework in which incidence, order, congruence, and continuity are separated into distinct groups of axioms. In the congruence group, the side-angle-side criterion (SAS) occupies a particularly important position: it is introduced as an axiom rather than derived from the preceding congruence axioms. Hilbert also proved that this axiom is independent of the remaining axioms of congruence by constructing a model in which the other axioms hold while SAS fails. \cite{Hilbert1950}

In modern terminology, following \cite{Hartshorne2000}, the system consisting of the incidence, betweenness, and congruence axioms is usually referred to as a \emph{Hilbert plane}. We denote this Hilbert plane by $\MG$, with \ref{C6} denoting the SAS triangle congruence axiom. $\MG^{-}$ is the Hilbert-plane system without SAS. This provides a natural setting for investigating the deductive role of SAS within $\MG^{-}$.

The question addressed here is the following: \emph{if SAS is removed from $\MG$, to what extent can it be recovered from the other classical triangle congruence criteria?} More precisely, using a purely synthetic approach we investigate whether each of the triangle congruence criteria can replace SAS in $\MG^{-}$, and which additional geometric principles are sufficient to reconstruct SAS. We adopt the usual notation ASA, SSS, and SAA for the corresponding
triangle congruence criteria.

This question is related to previous investigations by Donnelly in the metric framework of Birkhoff. In absolute geometry, he studied the replacement of SAS by SAA and SSS \cite{Donnelly2010,Donnelly2013,Donnelly2015}. He also investigated the non-redundancy of SAS in a non-continuous setting \cite{Donnelly2019}.

Other approaches have considered modifications of Hilbert's axioms from a different perspective. H\"{a}hl and Peters modify Hilbert's congruence axioms in order to develop a system closer to Euclid's approach, including an axiom related to the method of superposition, and prove its equivalence with Hilbert's original system. \cite{HaehlPeters2022} More recent work has continued to investigate alternative axiomatizations of absolute geometry and the role of the classical congruence principles. \cite{EdwardsPambuccian2026}

The purpose of the present paper is different. Rather than proposing a new foundational axiom system, we take the system $\MG^{-}$ as a fixed starting point and analyze, criterion by criterion, the additional principles that permit the reconstruction of SAS. The emphasis is therefore on the \emph{deductive dependencies} among the congruence criteria and the auxiliary geometric principles required in their reconstruction.

For ASA we prove that the principle concerning correspondence among rays inside congruent angles, denoted by [\textbf{RCT}], is sufficient to recover SAS. This principle in $\MG$ corresponds to Theorem 13 of \cite{Hilbert1950}. Thus, within $\MG^{-}$,

\begin{equation*}
\textrm{ASA, [\textbf{RCT}]}\;\vdash\;\textrm{SAS}.
\end{equation*}

For SSS and SAA we obtain corresponding reconstruction results, but with different additional principles. More precisely, we prove that
\begin{equation*}
\textrm{SSS},\;[\textbf{RCT}],\;Y\;\vdash\;\textrm{SAS},
\qquad
\textrm{SAA},\;[\textbf{RCT}],\;Z\;\vdash\;\textrm{SAS}.
\end{equation*}
Here, $Y$ is the set of principles consisting of the existence of a midpoint and the congruence of right triangles with congruent hypotenuses and acute angles, while $Z$ consists of the existence of angle bisectors, the congruence of supplementary angles, and the Pons Asinorum. We also investigate the independence of these additional principles by means of an explicit model. We find that the right-triangle congruence principle cannot be derived from the other auxiliary principles in the SSS case, and that the Pons Asinorum cannot be derived from the other auxiliary principles in the SAA case.

The resulting picture is not intended as a formal hierarchy of the congruence criteria. Rather, it provides a comparative analysis of their deductive strength relative to the same underlying system $\MG^{-}$. In particular, the fact that ASA requires only [\textbf{RCT}] for the reconstruction of SAS, whereas our reconstructions for SSS and SAA require additional auxiliary principles, suggests a meaningful qualitative distinction between the three criteria.

The analysis also places the classical results on triangle congruence in a common axiomatic framework. Standard treatments normally take SAS as a starting axiom and derive other congruence criteria from it. Here the direction of investigation is deliberately reversed: SAS is removed, and the other criteria are used as possible replacements. This reversal makes visible dependencies that are less apparent in the usual synthetic development and provides a systematic way of comparing alternative foundations for triangle congruence.

\section{From ASA to SAS}\label{PEN_s2}

We will follow the exposition of \cite{Greenberg1993}. We briefly recall the axioms and the standard results that
will be used in the sequel. The incidence, order (or betweenness) and congruence axioms are grouped into the three
families [\textbf{I}], [\textbf{O}] and [\textbf{C}], respectively. We
denote the three incidence axioms by [\textbf{I}1], [\textbf{I}2] and
[\textbf{I}3], and the four order axioms by [\textbf{O}1],
[\textbf{O}2], [\textbf{O}3] and [\textbf{O}4], following the order in
which they are presented in \cite{Greenberg1993}. We likewise denote the six
congruence axioms by [\textbf{C}1]--[\textbf{C}6], again following
the numbering of \cite{Greenberg1993}. The first three congruence axioms concern
the congruence of segments, [\textbf{C}4] and [\textbf{C}5] concern the
congruence of angles, while [\textbf{C}6] is the SAS criterion, which
connects the two congruence relations. We refer to \cite{Greenberg1993} for the complete statements of these axioms; for convenience, they are summarized in the final section. This constitutes the axiomatic system $\MG$, also called a Hilbert plane by \cite{Hartshorne2000}.

As anticipated, we introduce the axiomatic system $\MG^{-}$, consisting of the axiomatic system $\MG$ with SAS/\ref{C6} removed. From this point onward, unless otherwise indicated, the notation
\begin{equation*}
\Phi\vdash\Psi
\end{equation*}
will indicate that $\Psi$ is deducible from $\Phi$ in the axiomatic system $\MG^{-}$. In principle, any result previously proved in $\MG$ can be reinterpreted in the new setting by assuming SAS among the hypotheses. When possible, for brevity, we will state the theorem and refer to the proof given in the reference.

Let us begin with the Ray Correspondence Theorem, which corresponds to the following principle
\begin{enumerate}[label=\textrm{[}\textbf{RCT}\textrm{]}]
  \item\label{RCT} \emph{Given $\angle ABC \equiv \angle DEF$, then for every ray $\overrightarrow{BG}$ between $\overrightarrow{BA}$ and $\overrightarrow{BC}$, there exists a unique ray $\overrightarrow{EH}$ between $\overrightarrow{ED}$ and $\overrightarrow{EF}$ such that $\angle GBC \equiv \angle HEF$ and $\angle ABG \equiv \angle DEH$.}
\end{enumerate}
\begin{thm}\label{PEN_RCT}
  Assuming SAS, [\textbf{RCT}] holds.
\end{thm}
\begin{proof}
See Theorem 13 of \cite{Hilbert1950}.
\end{proof}
We first state the following preliminary result, whose proof does not require the congruence relations.
\begin{prop}\label{PEN_PO_08}
If $D$ is in the interior of $\angle CAB$, then:
\begin{enumerate}[label=(\roman*)]
  \item\label{PEN_PO_08_i} except for $A$, every point on the ray $\overrightarrow{AD}$ is in the interior of $\angle CAB$;
  \item no point on the opposite ray to $\overrightarrow{AD}$ is in the interior of $\angle CAB$.
\end{enumerate}
\end{prop}
\begin{proof}
See Proposition 3.8(a) and (b) of \cite{Greenberg1993}.
\end{proof}
The ordering relation between angles with the same vertex is already available, as it is essentially given by the inclusion of their interiors. The correspondence between rays induced by [\textbf{RCT}] allows this ordering to be transferred to angles with different vertices, provided that the containing angles are congruent. Thus, [\textbf{RCT}] allows us to compare angles in different locations while preserving the ordering relation under angle congruence \cite[Definition at p. 91]{Greenberg1993}.
\begin{prop}\label{PEN_CP12}\textsc{Ordering of Angles}\textrm{.}\\
If [\textbf{RCT}] is assumed, then, given the angles $\angle P$, $\angle Q$ and $\angle R$:
\begin{enumerate}[label=(\roman*)]
  \item \emph{Trichotomy}: exactly one of the following conditions holds
  \begin{equation*}
    \angle P<\angle Q,\quad \angle P\equiv\angle Q,\quad \angle Q<\angle P.
  \end{equation*}
  \item If $\angle P<\angle Q$ and $\angle Q\equiv\angle R$, then $\angle P <\angle R$.
  \item If $\angle Q<\angle P$ and $\angle Q\equiv\angle R$, then $\angle R<\angle P$.
  \item \emph{Transitivity}: If $\angle P<\angle Q$ and $\angle Q < \angle R$, then $\angle P<\angle R$.
\end{enumerate}
\end{prop}
\begin{proof}We prove the four statements with reference to Figure \ref{PEN_CP12_fn}:
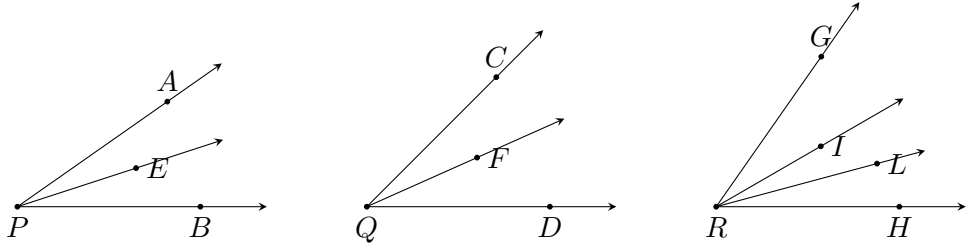
\begin{figure}[ht!]
  \centering
  \begin{tikzpicture}[>=stealth,scale=1.1]


\coordinate (P) at (0,0);

\draw[->] (P) -- ++(3,0);
\draw[->] (P) -- ++(35:3);

\coordinate (B) at ($(P)+(2.2,0)$);
\coordinate (A) at ($(P)+(35:2.2)$);

\filldraw (P) circle (0.8pt);
\filldraw (A) circle (0.8pt);
\filldraw (B) circle (0.8pt);

\node[below] at (P) {$P$};
\node[above] at (A) {$A$};
\node[below] at (B) {$B$};

\coordinate (E) at ($(P)+(18:1.5)$);

\draw[->] (P)--++(18:2.6);

\filldraw (E) circle (0.8pt);
\node[right] at (E) {$E$};


\coordinate (Q) at (4.2,0);

\draw[->] (Q) -- ++(3,0);
\draw[->] (Q) -- ++(45:3);

\coordinate (D) at ($(Q)+(2.2,0)$);
\coordinate (C) at ($(Q)+(45:2.2)$);

\filldraw (Q) circle (0.8pt);
\filldraw (C) circle (0.8pt);
\filldraw (D) circle (0.8pt);

\node[below] at (Q) {$Q$};
\node[above] at (C) {$C$};
\node[below] at (D) {$D$};

\coordinate (F) at ($(Q)+(24:1.45)$);

\draw[->] (Q)--++(24:2.6);

\filldraw (F) circle (0.8pt);
\node[right] at (F) {$F$};


\coordinate (R) at (8.4,0);

\draw[->] (R) -- ++(3,0);
\draw[->] (R) -- ++(55:3);

\coordinate (H) at ($(R)+(2.2,0)$);
\coordinate (G) at ($(R)+(55:2.2)$);

\filldraw (R) circle (0.8pt);
\filldraw (G) circle (0.8pt);
\filldraw (H) circle (0.8pt);

\node[below] at (R) {$R$};
\node[above] at (G) {$G$};
\node[below] at (H) {$H$};

\coordinate (I) at ($(R)+(30:1.45)$);

\draw[->] (R)--++(30:2.6);

\filldraw (I) circle (0.8pt);
\node[right] at (I) {$I$};
\coordinate (L) at ($(R)+(15:2.0)$);

\draw[->] (R)--++(15:2.6);

\filldraw (L) circle (0.8pt);
\node[right] at (L) {$L$};
\end{tikzpicture}
  \caption{Ordering of Angles}\label{PEN_CP12_fn}
\end{figure}
\begin{enumerate}[label=(\roman*)]
  \item If $\angle P\equiv\angle Q$, the claim follows. Suppose that $\angle P \not\equiv \angle Q$. Let $\angle P=\angle APB$ and $\angle Q=\angle CQD$. Then, by \ref{C4}, there exists a unique ray $\overrightarrow{PE}$ which, together with $\overrightarrow{PA}$, lies on the same side of $\overleftrightarrow{PB}$ and such that $\angle EPB\equiv\angle CQD$. The point $E$ may be in the interior of $\angle APB$, or $A$ may be in the interior of $\angle EPB$. Equivalently, using Proposition \ref{PEN_PO_08}(i), $\overrightarrow{PE}$ is between $\overrightarrow{PA}$ and $\overrightarrow{PB}$, or $\overrightarrow{PA}$ is between $\overrightarrow{PE}$ and $\overrightarrow{PB}$. If $\overrightarrow{PE}$ is between $\overrightarrow{PA}$ and $\overrightarrow{PB}$, then by the definition of $<$, $\angle CQD<\angle APB$, which is the claim. If instead $\overrightarrow{PA}$ is between $\overrightarrow{PE}$ and $\overrightarrow{PB}$, we can apply [\textbf{RCT}]. Since $\angle EPB\equiv\angle CQD$, there exists $\overrightarrow{QF}$ between $\overrightarrow{QC}$ and $\overrightarrow{QD}$ such that $\angle FQD\equiv\angle APB$, that is, $\angle APB<\angle CQD$, which is the claim and completes the proof. The fact that [\textbf{RCT}] uniquely determines the corresponding ray ensures that only one of the two inequalities can hold.
  \item Using the same notation as in (i) for the angles $\angle P$ and $\angle Q$, let $\angle R=\angle GRH$. By hypothesis, there exists a ray $\overrightarrow{QF}$ between $\overrightarrow{QC}$ and $\overrightarrow{QD}$ such that $\angle FQD\equiv\angle APB$. Using the hypothesis $\angle Q\equiv\angle R$ and [\textbf{RCT}], we obtain a ray $\overrightarrow{RI}$ between $\overrightarrow{RG}$ and $\overrightarrow{RH}$ such that $\angle IRH\equiv\angle FQD$. By \ref{C2}, $\angle IRH\equiv\angle APB$, and therefore, by the definition of $<$, the claim follows.
  \item Analogously to (ii), exchanging $\angle P$ and $\angle Q$.
  \item Using the same notation as in (ii), by the same reasoning we can obtain the ray $\overrightarrow{QF}$ between $\overrightarrow{QC}$ and $\overrightarrow{QD}$ such that $\angle FQD\equiv\angle APB$. Similarly, from $\angle Q < \angle R$ we obtain a ray $\overrightarrow{RI}$ between $\overrightarrow{RG}$ and $\overrightarrow{RH}$ such that $\angle IRH\equiv\angle CQD$. Applying [\textbf{RCT}], there exists $\overrightarrow{RL}$ between $\overrightarrow{RI}$ and $\overrightarrow{RH}$ such that $\angle LRH\equiv\angle FQD$. By convexity of the angle interior \cite[Theorem 4.20]{Lee2013}, we have that $\overrightarrow{RL}$ is between $\overrightarrow{RG}$ and $\overrightarrow{RH}$; by \ref{C5}, $\angle LRH\equiv\angle APB$. It follows that $\angle APB<\angle GRH$.
\end{enumerate}
\end{proof}
\begin{figure}
  \centering
  \begin{tikzpicture}[scale=0.85]
  \coordinate (B) at (0,0);
  \coordinate (C) at (5,0.5);
  \coordinate (A) at (1.75,2.75);

  \coordinate (E) at (8,0);
  \coordinate (F) at (13,0.5);
  \coordinate (D) at (9.75,2.75);

  \draw[very thick] (A) -- (B) -- (C) -- cycle;
  \draw[very thick] (D) -- (E) -- (F) -- cycle;

  \fill (A) circle (1.5pt) node[above] {$A$};
  \fill (B) circle (1.5pt) node[left] {$B$};
  \fill (C) circle (1.5pt) node[right] {$C$};

  \fill (D) circle (1.5pt) node[above] {$D$};
  \fill (E) circle (1.5pt) node[left] {$E$};
  \fill (F) circle (1.5pt) node[right] {$F$};

  \PerpMark{A}{B}{0.5}{0.075}{thin}
  \PerpMark{D}{E}{0.5}{0.075}{thin}
  \PerpMark{A}{C}{0.5}{0.075}{thin}
  \PerpMark{A}{C}{0.525}{0.075}{thin}
  \PerpMark{D}{F}{0.51}{0.075}{thin}
  \PerpMark{D}{F}{0.525}{0.075}{thin}

  \pic[draw, angle radius=0.55cm] {angle=B--A--C};
  \pic[draw, angle radius=0.55cm] {angle=E--D--F};

\coordinate (H) at ($(D)!0.80!(F)$);
\fill (H) circle (1.5pt) node[above] {$\quad H$};
\draw[densely dotted] (E) -- (H);
\pic[draw, angle radius=0.85cm] {angle=H--E--D};
\pic[draw, angle radius=0.95cm] {angle=H--E--D};

\pic[draw, angle radius=0.85cm] {angle=C--B--A};
\pic[draw, angle radius=0.95cm] {angle=C--B--A};
\end{tikzpicture}
  \caption{\cite[Theorem 16.5]{Martin1982} Schematic}\label{PEN_ASAf}
\end{figure}
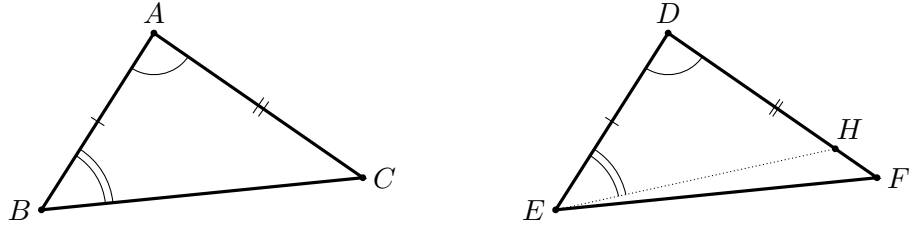
In \cite[Theorem 16.5]{Martin1982}, within a metric framework, it is shown that ASA implies SAS by making
use of the trichotomy of the order on angles. Indeed (see Figure \ref{PEN_ASAf}), let $\triangle ABC$ and $\triangle DEF$ be such that:
\begin{equation*}
AB\equiv DE,\quad \angle BAC\equiv \angle EDF,\quad AC\equiv DF.
\end{equation*}
The argument consists in separately excluding the cases
\begin{equation*}
\angle ABC<\angle DEF
\qquad\text{and}\qquad
\angle DEF<\angle ABC,
\end{equation*}
and then concluding, by trichotomy, that necessarily
\begin{equation*}
\angle ABC\equiv\angle DEF.
\end{equation*}
Also within the synthetic development of the theory followed in this text,
it is possible to exclude the two order relations separately. For example, suppose that
\begin{equation*}
\angle ABC<\angle DEF.
\end{equation*}
By definition of $<$, there exists a ray $\overrightarrow{EG}$ internal to $\angle DEF$ such that
\begin{equation*}
\angle ABC\equiv\angle DEG.
\end{equation*}
Applying \cite[Crossbar Theorem, p.~82]{Greenberg1993}, this ray meets the segment $DF$ at a point $H$. The ASA criterion then gives
\begin{equation*}
\triangle DEH\equiv\triangle ABC,
\end{equation*}
from which it follows that $DH\equiv AC$. Moreover, since $D-H-F$, we obtain $DH<DF$, and hence $AC<DF$, in contradiction with the hypothesis $AC\equiv DF$ and with the trichotomy of the order relation between segments \cite[Proposition 3.13]{Greenberg1993}. Therefore,
\begin{equation*}
\angle ABC<\angle DEF
\end{equation*}
cannot hold.

By an entirely analogous argument, the relation
\begin{equation*}
\angle DEF<\angle ABC
\end{equation*}
can also be excluded.

The difference between the metric approach of \cite{Martin1982} and the synthetic approach adopted in this article, however, emerges in the next step. If we use ASA as a replacement axiom for SAS, we must first prove the trichotomy of the order relation between angles without making use of SAS. In the proof of Proposition \ref{PEN_CP12}(i), essential use is made of [\textbf{RCT}], according to which the rays internal to two congruent angles are in a one-to-one correspondence preserving the congruence of the angles that they delimit. The proof of [\textbf{RCT}], in turn, makes use of the SAS/\ref{C6} criterion. Therefore, the argument of \cite{Martin1982} cannot be transferred directly to the present development without introducing a circular dependency.

The preceding analysis confirms the central role of the SAS criterion in the structure of the theory. Axioms \ref{C1}--\ref{C3} describe exclusively the congruence of segments, whereas axioms \ref{C4} and \ref{C5} concern only the congruence of angles. The SAS criterion is the first result that establishes a relation between these two structures, allowing information to be transferred from one to the other through triangle congruence. In this sense, it represents the true link between the two structures underlying the entire theory of congruence. The following main result also holds:
\begin{thm}\label{PEN_ASA}
If the ASA criterion and [\textbf{RCT}] are assumed, then SAS.
\end{thm}
\begin{proof}
Let $\triangle ABC$ and $\triangle DEF$ be such that:
\begin{equation*}
AB\equiv DE,\quad \angle BAC\equiv \angle EDF,\quad AC\equiv DF.
\end{equation*}
Using \ref{C4}, we can construct a ray $\overrightarrow{BG}$ on the same side of $C$ with respect to $\overleftrightarrow{AB}$ such that:
\begin{equation*}
\angle ABG\equiv\angle DEF.
\end{equation*}
We have established in the reasoning preceding the theorem that $\angle DEF\nless\angle ABC$, and therefore either $\overrightarrow{BC}=\overrightarrow{BG}$, or $\overrightarrow{BC}$ is internal to the angle $\angle ABG$. In the latter case, applying [\textbf{RCT}], we know that there exists a ray $\overrightarrow{EH}$ such that $\angle DEH\equiv\angle ABC$, and hence $\angle ABC<\angle DEF$, which contradicts $\angle ABC\nless\angle DEF$. Therefore $\overrightarrow{BC}=\overrightarrow{BG}$, that is,
\begin{equation*}
\angle ABC\equiv\angle DEF,
\end{equation*}
and applying ASA we obtain the congruence of the two triangles.
\end{proof}
\section{From SSS or SAA to SAS}\label{PEN_s3}
We can now ask whether, by adding one of the criteria SAA and SSS to the axiomatic system $\MG^{-}$, it is possible to reconstruct SAS, possibly also assuming some auxiliary principles. As in the case of ASA, we require these principles to be provable in the axiomatic system $\MG$. We shall carry out the analysis for the SAA and SSS criteria in parallel and introduce the following abbreviated notation for the principles that will be used below.
\begin{enumerate}[label=\textrm{[}\textbf{SA}\textrm{]}]
  \item\label{SA} \emph{Supplements of congruent angles are congruent.}
\end{enumerate}
\begin{enumerate}[label=\textrm{[}\textbf{MS}\textrm{]}]
  \item\label{MS} \emph{Every segment has a unique midpoint.}
\end{enumerate}
\begin{enumerate}[label=\textrm{[}\textbf{AB}\textrm{]}]
  \item\label{AB} \emph{Every angle has a bisector.}
\end{enumerate}
\begin{enumerate}[label=\textrm{[}\textbf{HA}\textrm{]}]
  \item\label{HA} \emph{Two right triangles having congruent hypotenuses and congruent acute angles, with the right angle opposite the acute angle, are congruent.}
\end{enumerate}
\begin{rem}\label{PEN_REM_HA}
\ref{HA} can be proved in $\MG$ as a particular case of the SAA criterion using the congruence of all right angles. When, instead, \ref{HA} is assumed as an additional principle to $\MG^{-}$, this result is not available, and it is therefore necessary to specify to which of the two supplements the right angle is congruent. We consider the supplement whose common side is the side containing the leg opposite the acute angle. In Figure \ref{PEN_REM_HA_f}, the relevant supplement is shown by a dotted line.
\begin{figure}[ht!]
\centering
\begin{tikzpicture}[scale=1,thin]
\coordinate (A) at (0,0);
\coordinate (Q) at (4,0);
\coordinate (Qprime) at (5,0);

\fill (A) circle (1.5pt) node[below] {$A$};
\fill (Q) circle (1.5pt) node[below] {$Q$};

\coordinate (P) at (4.0,2);

\draw[semithick] (P) -- (Q);
\draw[semithick] (A) -- (P);
\draw[semithick] (A) -- (Q);
\draw[thin] (Q) -- (Qprime);

\fill (P) circle (1.5pt) node[above] {$P$};
\RightAngle{Q}{A}{P}{0.10}{0.15}
\pic[draw, angle radius=0.75cm] {angle=Q--A--P};
\tikzset{every path/.style={densely dotted}}
\RightAngle{Q}{Qprime}{P}{0.40}{0.15}
\end{tikzpicture}
\caption{\ref{HA} Schematic}\label{PEN_REM_HA_f}
\end{figure}
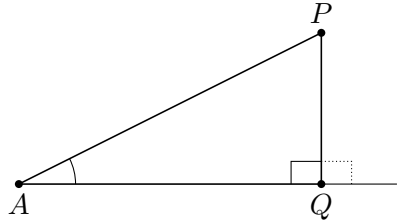
\end{rem}
Assuming [\textbf{SA}] and [\textbf{RCT}] as additional principles to $\MG^{-}$, we can reformulate the classical result corresponding to Theorem 15 of \cite{Hilbert1950}.
\begin{thm}\label{PEN_EU4P}
If [\textbf{SA}] and [\textbf{RCT}] are assumed, then all right angles are congruent to one another.
\end{thm}
\begin{proof}
Let $\angle BAD$ be congruent to its supplementary angle $\angle CAD$, and likewise let $\angle B'A'D'$ be congruent to its supplementary angle $\angle C'A'D'$. Hence $\angle BAD$, $\angle CAD$, $\angle B'A'D'$, and $\angle C'A'D'$ are all right angles.

Assume, to the contrary, that
\begin{equation*}
\angle B'A'D'\not\equiv\angle BAD.
\end{equation*}
\begin{figure}[ht!]
\centering
\begin{tikzpicture}[scale=1,thick]
\coordinate [label=below:$A$] (A) at (3,0);
\coordinate [label=below right:$C$] (C) at (5,0);
\coordinate [label=below:$B$] (B) at (1,0);
\coordinate [label=above:$D$] (D) at (3,3);
\coordinate [label=above:$D''$] (J) at ( 2,3);
\coordinate [label=above:$D'''$] (K) at (4,3);
\draw [-] (A) -- (D);
\draw [-] (A) -- (B);
\draw [-] (A) -- (C);
\draw [-,thin,dashed] (A) -- (J);
\draw [-,thin,dashed] (A) -- (K);
\coordinate [label=below:$A'$] (E) at (3+5,0);
\coordinate [label=below:$\;$] (E11) at (3+5+0.25,0+0.25);
\coordinate [label=below:$\;$] (E12) at (3+5,0+0.25);
\coordinate [label=below:$\;$] (E21) at (3+5+0.25,0);
\coordinate [label=below right:$C'$] (G) at (5+5,0);
\coordinate [label=below:$B'$] (F) at (1+5,0);
\coordinate [label=above:$D'$] (H) at (3+5,3);
\draw [-,thin] (E11) -- (E12);
\draw [-,thin] (E11) -- (E21);
\draw [-] (E) -- (H);
\draw [-] (E) -- (F);
\draw [-] (E) -- (G);
\end{tikzpicture}
\caption{ Theorem \ref{PEN_EU4P} Schematic}
\label{PEN_f_CP14}
\end{figure}
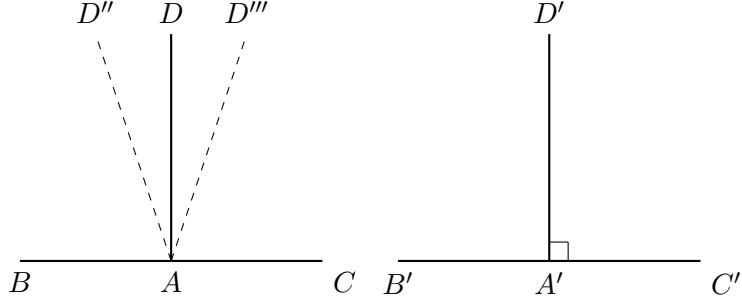
Lay off the angle $\angle B'A'D'$ upon the ray $\overrightarrow{AB}$ in such a manner that the ray $\overrightarrow{AD''}$ arising from this construction falls either within the angle $BAD$ or within the angle $CAD$. Suppose, for example, that $\overrightarrow{AD''}$ falls within the angle $BAD$. Since
\begin{equation*}
\angle B'A'D'\equiv\angle BAD'',
\end{equation*}
[\textbf{SA}] gives
\begin{equation*}
\angle C'A'D'\equiv\angle CAD''.
\end{equation*}
Moreover,
\begin{equation*}
\angle B'A'D'\equiv\angle C'A'D',
\end{equation*}
and hence, by [\textbf{C5}],
\begin{equation*}
\angle BAD''\equiv\angle CAD''.
\end{equation*}

Furthermore, since
\begin{equation*}
\angle BAD\equiv\angle CAD,
\end{equation*}
[\textbf{RCT}] yields a ray $\overrightarrow{AD'''}$ lying within the angle $CAD$ such that
\begin{equation*}
\angle BAD''\equiv\angle CAD'''
\qquad\text{and}\qquad
\angle DAD''\equiv\angle DAD'''.
\end{equation*}
The angle $\angle BAD''$ is congruent to $\angle CAD''$, as shown above, and is also congruent to $\angle CAD'''$. Therefore, by [\textbf{C5}],
\begin{equation*}
\angle CAD''\equiv\angle CAD'''.
\end{equation*}
This is impossible, however, because $\overrightarrow{AD''}$ and $\overrightarrow{AD'''}$ lie on the same side of the line $\overleftrightarrow{AC}$, and [\textbf{C4}] guarantees the uniqueness of the ray that realizes a given angle on a prescribed side of a given line.

The case in which $\overrightarrow{AD''}$ falls within the angle $CAD$ is analogous. Hence the assumption
\begin{equation*}
\angle B'A'D'\not\equiv\angle BAD
\end{equation*}
is impossible. Therefore
\begin{equation*}
\angle B'A'D'\equiv\angle BAD.
\end{equation*}
Thus all right angles are congruent to one another.
\end{proof}
A second classical result is that SSS proves the Pons asinorum:
\begin{thm}\label{PEN_LLL_PA} If we assume SSS, then [\textbf{PA}] follows.
\end{thm}
\begin{proof}
Let $\triangle ABC$ be such that
\begin{equation*}
AB\equiv AC.
\end{equation*}
Consider the same triangle with the vertices $B$ and $C$ interchanged,
namely $\triangle ACB$. The corresponding sides of the two triangles are
\begin{equation*}
AB\equiv AC,\qquad AC\equiv AB,\qquad BC\equiv CB.
\end{equation*}
Therefore, by the SSS criterion,
\begin{equation*}
\triangle ABC\equiv\triangle ACB.
\end{equation*}
The correspondence of the angles gives
\begin{equation*}
\angle ABC\equiv\angle ACB,
\end{equation*}
which is precisely the conclusion of the Pons asinorum.
\end{proof}
Let us now connect SAA to \ref{HA}.
\begin{thm}\label{PEN_LAA_HA}
If SAA and [\textbf{RCT}], \ref{SA} are assumed, then \ref{HA} follows.
\end{thm}
\begin{proof}
Let $\triangle ABC$ and $\triangle DEF$ be two right triangles such that
\begin{equation*}
AB\equiv DE,\qquad \angle BAC\equiv\angle EDF,
\end{equation*}
where $AB$ and $DE$ are the hypotenuses and the right angles are opposite the acute angles under consideration. By [\textbf{RCT}] and [\textbf{SA}], it follows from Theorem \ref{PEN_EU4P} that all right angles are congruent. In particular,
\begin{equation*}
\angle ACB\equiv\angle DFE.
\end{equation*}
Therefore, by SAA,
\begin{equation*}
\triangle ABC\equiv\triangle DEF,
\end{equation*}
and hence \ref{HA} holds.
\end{proof}
The following two results establish the existence of a right angle when one of the two criteria SSS and SAA is assumed, together with the principle \ref{MS} and, respectively, \ref{AB} and [\textbf{PA}].
\begin{thm}\label{PEN_LLL_RAE} If we assume SSS and \ref{MS}, then we can deduce the following results:
\begin{enumerate}[label=(\roman*)]
  \item Every angle has a bisector, i.e. \ref{AB}.
  \item A right angle exists.
\end{enumerate}
\end{thm}
\begin{proof}
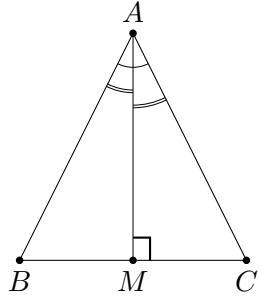
\begin{figure}[ht!]
\centering
\begin{tikzpicture}[scale=1,thick]
\coordinate [label=below:$B$] (B) at (0,0);
\coordinate [label=above:$A$] (A) at (1.5,3);
\coordinate [label=below:$C$] (C) at (3,0);
\coordinate [label=below:$M$] (M) at (1.5,0);
\draw plot [mark=*,mark size=1pt] coordinates {(B)};
\draw plot [mark=*,mark size=1pt] coordinates {(A)};
\draw plot [mark=*,mark size=1pt] coordinates {(C)};
\draw plot [mark=*,mark size=1pt] coordinates {(M)};
\draw [-,very thin] (A) -- (M);
\draw [-,thin] (B) -- (A);
\draw [-,thin] (A) -- (C);
\draw [-,thin] (C) -- (B);
\RightAngle{M}{A}{C}{0.10}{0.15}
\QuoteAngle{A}{B}{C}{0.45}{}{very thin}
\QuoteAngle{A}{B}{M}{0.75}{}{very thin}
\QuoteAngle{A}{B}{M}{0.785}{}{very thin}
\QuoteAngle{A}{M}{C}{0.95}{}{very thin}
\QuoteAngle{A}{M}{C}{0.985}{}{very thin}
\end{tikzpicture}
\caption{Theorem \ref{PEN_LLL_RAE}-Proof schema}
\label{PEN_LLL_RAE_f}
\end{figure}
With reference to Figure \ref{PEN_LLL_RAE_f}, let the angle $\angle BAC$ be given. Without loss of generality, we can always reduce the argument to the case $AB\equiv AC$ using \ref{C1}. Consider the isosceles triangle $\triangle ABC$, and let $M$ be the midpoint of $BC$. Then the two triangles $\triangle ABM$ and $\triangle ACM$ have their corresponding sides congruent in pairs. Therefore, we can apply SSS and conclude that the two triangles are congruent. It follows that $\angle BAM\equiv \angle CAM$, so $\overrightarrow{AM}$ is a bisector of $\angle BAC$. Moreover, $\angle AMB\equiv \angle AMC$, and therefore the angle at $M$ is right.
\end{proof}
\begin{thm}\label{PEN_LAA_RAE} If we assume SAA, [\textbf{PA}], and \ref{AB}, then the base of an isosceles triangle is intersected perpendicularly by the bisector of the opposite angle, and the intersection point is the midpoint of the base segment.
\end{thm}
\begin{proof}
Consider the isosceles triangle $\triangle ABC$ with $AB\equiv AC$. Let $\overrightarrow{AM}$ be the bisector of $\angle BAC$, where $M$ denotes the intersection point with $BC$, whose existence follows from the Crossbar Theorem. Consider the two triangles $\triangle ABM$ and $\triangle ACM$. Then $AM\equiv AM$, $\angle BAM\equiv\angle CAM$, and, by applying [\textbf{PA}] to $\triangle ABC$, we have $\angle ABC\equiv\angle ACM$. Therefore, applying SAA, we obtain $\triangle ABM\equiv\triangle ACM$. It follows, in particular, that $\angle AMB\equiv\angle AMC$, that is, the angle is right, and $BM\equiv CM$; therefore, $M$ is the midpoint of $BC$.
\end{proof}
The uniqueness of the angle bisector is obtained by assuming [\textbf{RCT}]:
\begin{thm}\label{PEN_RCT_AB}
If we assume [\textbf{RCT}], then any angle that has a bisector has a unique bisector.
\end{thm}
\begin{proof}
Consider the angle $\angle BAC$ and let $\overrightarrow{AQ}$ and $\overrightarrow{AM}$ be two of its bisectors. Then $\angle BAQ\equiv\angle CAQ$ and $\angle BAM\equiv\angle CAM$. [\textbf{RCT}] allows us to establish the properties of the order relation $<$ without directly using SAS/\ref{C6} (see Proposition \ref{PEN_CP12}). Suppose, therefore, that
\begin{equation*}
  \angle BAQ<\angle BAM.
\end{equation*}
By the definition of $<$, $\overrightarrow{AQ}$ is interior to $\angle BAM$, which implies that $\overrightarrow{AM}$ is interior to $\angle CAQ$, that is,
\begin{equation*}
  \angle CAM<\angle CAQ.
\end{equation*}
Since congruent angles can be substituted in an inequality, we obtain
\begin{equation*}
  \angle BAM<\angle BAQ
\end{equation*}
which contradicts the trichotomy of the order relation. Conversely, assuming $\angle BAM<\angle BAQ$ would, by analogous reasoning, yield $\angle BAQ<\angle BAM$. Therefore, $\overrightarrow{AQ}=\overrightarrow{AM}$, that is, the angle bisector is unique.
\end{proof}
Using \ref{HA}, we can prove the following result.
\begin{thm}\label{PEN_HA}
If we assume \ref{HA}, then two lines intersecting a line perpendicularly on the same side are parallel.
\end{thm}
\begin{proof}
\begin{figure}[ht!]
\centering
\begin{tikzpicture}[scale=1,thin]
\coordinate (A) at (0,0);
\coordinate (Q) at (4,0);
\coordinate (Qprime) at (4.5,0);

\fill (A) circle (1.5pt) node[below] {$A$};
\fill (Q) circle (1.5pt) node[below] {$Q$};
\fill (Qprime) circle (1.5pt) node[below] {$Q'$};

\coordinate (P) at (4.15,2);

\draw[semithick] (P) -- (Q);
\draw[semithick] (P) -- (Qprime);
\draw[semithick] (A) -- (P);
\draw[semithick] (A) -- (Qprime);

\fill (P) circle (1.5pt) node[above] {$P$};
\RightAngle{Q}{A}{P}{0.10}{0.15}
\RightAngle{Qprime}{A}{P}{0.08}{0.15}
\end{tikzpicture}
\caption{Theorem \ref{PEN_HA}}\label{PEN_HA_f}
\end{figure}
We first note that it is important to specify ``on the same side'' in order to consider the same pair of supplementary angles as right angles. Suppose, for contradiction, that there exist two perpendiculars through $P$, and denote by $Q$ and $Q'$ their respective intersections with $l$. If we consider a point $A$ on $l$ outside the segment $QQ'$, then, as shown in Figure \ref{PEN_HA_f}, we have two right triangles $\triangle PQA$ and $\triangle PQ'A$ that are congruent by \ref{HA}. Indeed, since $A$ lies outside the segment $QQ'$, either $A-Q-Q'$ or $A-Q'-Q$, and hence $\angle PAQ\equiv\angle PAQ'$. It follows that $AQ\equiv AQ'$, with $Q\neq Q'$, which contradicts \ref{C1}.
\end{proof}
If we use together \ref{HA} and \ref{RCT} we can prove a particular case of the exterior angle theorem for right angle:
\begin{thm}\label{PEN_HA_EA}
If \ref{HA}, \ref{RCT} hold, then:
\begin{enumerate}[label=\textrm{[}\textbf{EA}\textrm{]}]
  \item If an angle of a triangle is congruent to its supplement with respect one of its side, then the angle opposite to this side is acute.
\end{enumerate}
\end{thm}
\begin{proof}
Let $\triangle ABC$ be a triangle such that $\angle BAC$ is congruent to its supplement with respect to $AC$, let $D$ such that $D-A-B$, we have $\angle CAD\equiv\angle CAB$.
Let $F$ be a point such that $A-B-F$, by \ref{C4}, construct the ray $\overrightarrow{BE}$ such that $\angle ABE\equiv\angle CAB$, and $E$ and $C$ lie on the same side of $\overleftrightarrow{AB}$. Note that $\overrightarrow{BE}$ and $\overrightarrow{AC}$ are perpendicular to $\overleftrightarrow{AB}$.  By Theorem \ref{PEN_HA} $\overrightarrow{BE}$ and $\overrightarrow{AC}$ cannot intersect otherwise they should coincide and $A$, $B$ and $C$ are on the same line. This also rules out $\angle CAB\equiv \angle CBA$: in that case $\angle CBA$ would also be right, so $\overrightarrow{CA}$ and $\overrightarrow{CB}$ would be two perpendiculars to $\overleftrightarrow{AB}$ through the common external point $C$; by the argument of Theorem \ref{PEN_HA} (applied with $C$ in place of $P$), their feet would coincide, forcing $A=B$, contrary to hypothesis (note that this uses only the hypothesis $\angle CAD\equiv\angle CAB$ that $\angle CAB$ itself is right, so \ref{SA} is not needed here).
If $\angle CAB< \angle CBA$, then $\overrightarrow{BC}$  and is inside $\angle EBF$. Then $A$ and $C$ are on the opposite side of $\overleftrightarrow{BE}$. Otherwise since $C$ and $F$ are on the same side of $\overleftrightarrow{BE}$, then by \ref{O4}(i) also $A$ and $F$ should be on same side, in contradiction with $A-B-F$. It follows that $\overrightarrow{BE}$ and $\overrightarrow{AC}$ intersect and this contradicts Theorem \ref{PEN_HA}. Hence by \ref{RCT} trichotomy holds and then  $\angle CAB>\angle CBA$.
\end{proof}
We are now able to establish the existence and uniqueness of the orthogonal projection of a point onto a line.
\begin{thm}\label{PEN_LLL_PO}
If the principles of one of the following groups are assumed:
\begin{enumerate}[label=(\roman*)]
\item SSS, [\textbf{RCT}], \ref{MS} and \ref{HA};
\item SAA, [\textbf{RCT}], \ref{AB} and [\textbf{PA}].
\end{enumerate}
Then there exists a unique perpendicular to the line $l$ through a point $P$ external to it. That is, there exists a unique orthogonal projection of $P$ onto $l$.
\end{thm}
\begin{proof}
\begin{figure}[ht!]
\centering
\begin{tikzpicture}[scale=1,thick]
\coordinate [label=left:$A$] (A) at (0,0);
\coordinate [label=above :$r\quad$] (X) at (6,-3);
\coordinate [label=below right:$P$] (P) at (3,1.5);
\coordinate [label=below right:$Q$] (Q) at (3,0);
\coordinate [label=below left:$B$] (B) at (6,0);
\coordinate [label=below left:$l$] (l) at (9,0);
\coordinate [label=above right:$P'$] (P') at (3,-1.5);
\draw plot [mark=*,mark size=1pt] coordinates {(A)};
\draw plot [mark=*,mark size=1pt] coordinates {(P)};
\draw plot [mark=*,mark size=1pt] coordinates {(Q)};
\draw plot [mark=*,mark size=1pt] coordinates {(P')};
\draw plot [mark=*,mark size=1pt] coordinates {(B)};
\draw [-latex,thin] (A) -- (l);
\draw [-,thin] (A) -- (X);
\draw [-,thin] (P) -- (P');
\draw [-,thin] (A) -- (P);
\coordinate [label=below right:$\:$] (P1) at (1.5,0.75+0.1);
\coordinate [label=below right:$\:$] (P2) at (1.5,0.75-0.1);
\draw [-,thin] (P1) -- (P2);
\coordinate [label=below right:$\:$] (P1') at (1.5,-0.75+0.1);
\coordinate [label=below right:$\:$] (P2') at (1.5,-0.75-0.1);
\draw [-,thin] (P1') -- (P2');
\draw [-,thin] (1,0) arc (0:{atan(1.5/3)}:1);
\draw [-,thin] (1,0) arc (0:{-atan(1.5/3)}:1);
\end{tikzpicture}
\caption{Theorem \ref{PEN_LLL_PO}-Proof schematic}
\label{PEN_LLL_PO_f}
\end{figure}
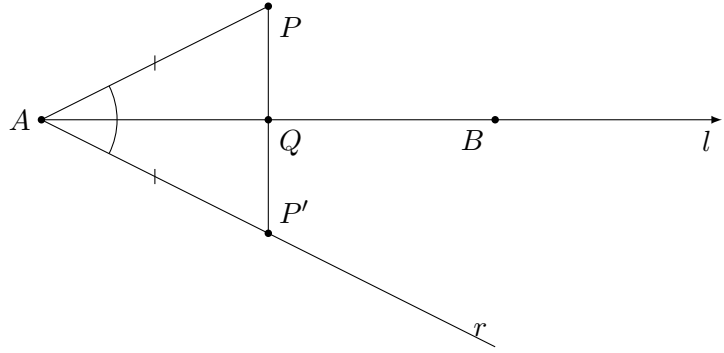
By hypothesis, $P$ does not lie on $l$. Let $A$ and $B$ be two distinct points of $l$. If one of the two angles $\angle PAB$ and  $\angle PBA$ is right, the result follows immediately. Suppose that both are non-right. On the side opposite to $P$ with respect to $l$, using \ref{C4}, draw from $A$ a ray $r$ that forms with $l$ an angle congruent to $\angle PAB$. Using \ref{C1}, there exists a point $P'$ on $r$ such that $AP\equiv AP'$. We therefore have the isosceles triangle $\triangle APP'$, and hence $\angle APP'\equiv\angle AP'P$, since in (i) SSS implies [\textbf{PA}] (Theorem \ref{PEN_LLL_PA}), whereas in (ii) [\textbf{PA}] is directly an assumption. We now treat (i) and (ii) separately:
\begin{enumerate}[label=(\roman*)]
\item Let $Q$ be the midpoint of $PP'$. Then $\triangle APQ\equiv \triangle AP'Q$ by SSS, and in particular $\angle PAQ\equiv \angle P'AQ$. By Theorem \ref{PEN_RCT_AB} on the uniqueness of the angle bisector, it follows that $\overleftrightarrow{AQ}=l$. Moreover, $\angle PQA\equiv \angle P'QA$, and hence the angle is right. Therefore, the line $\overleftrightarrow{PP'}$ intersects $l$ orthogonally.
\item By construction, $l$ is the angle bisector passing through the midpoint $Q$ of $PP'$, and $\angle PQA\equiv \angle P'QA$, so the angle is right (Theorem \ref{PEN_LAA_RAE}).
\end{enumerate}
Therefore, the line $\overleftrightarrow{PP'}$ intersects $l$ orthogonally. Uniqueness follows directly by applying Theorem \ref{PEN_HA}.
\end{proof}
We are therefore able to prove the following main theorem.
\begin{thm}\label{PEN_LLL_LAL}
If the principles of one of the following groups are assumed to hold:
\begin{enumerate}[label=(\roman*)]
\item SSS, [\textbf{RCT}], \ref{MS} and \ref{HA};
\item SAA, [\textbf{RCT}], \ref{AB}, \ref{SA} and [\textbf{PA}].
\end{enumerate}
Then SAS follows.
\end{thm}
\begin{proof}
Let the two triangles $\triangle ABC$ and $\triangle DEF$ be given such that
\begin{equation*}
AB \equiv DE,\quad AC\equiv DF,\quad\angle BAC\equiv\angle EDF.
\end{equation*}
We want to prove $\triangle ABC\equiv\triangle DEF$.

We observe that, by Theorem \ref{PEN_LAA_HA}, in case (ii) we can deduce \ref{HA} from SAA, [\textbf{RCT}], \ref{SA}.

Draw the perpendicular through $A$ to side $BC$ and let $P$ be its foot. Draw the
perpendicular through $D$ to side $EF$ and let $Q$ be its foot.

If $\angle PAC\equiv\angle QDF$, then, by hypothesis $AC\equiv DF$, applying \ref{HA}
gives $\triangle PAC\equiv\triangle QDF$ and, in particular, $PC\equiv QF$ and
$\angle ACB\equiv\angle DFE$.

Let us analyze the point ordering, taking into account the hypotheses
$\angle PAC\equiv\angle QDF$ and $\angle BAC\equiv\angle EDF$:
\begin{enumerate}[label=(\arabic*)]
  \item If $B-P-C$, then $\angle PAB<\angle BAC$, because $\overrightarrow{AP}$ is
  inside $\angle BAC$. By [\textbf{RCT}], $\overrightarrow{DQ}$ is inside
  $\angle EDF$, hence $E-Q-F$.
  \item If $P-B-C$, then $\angle PAB>\angle BAC$, because $\overrightarrow{AB}$ is
  inside $\angle PAC$. By [\textbf{RCT}], $\overrightarrow{DE}$ is inside
  $\angle QDF$, hence $Q-E-F$.
\end{enumerate}
By $\angle BAC\equiv\angle EDF$ and $\angle PAC\equiv\angle QDF$, applying
\ref{RCT}, in both cases (1) and (2) we have $\angle PAB\equiv\angle QDE$, and by
hypothesis $AB\equiv DE$, applying \ref{HA} gives $\triangle PAB\equiv\triangle QDE$,
and, in particular, $BP\equiv EQ$ and $\angle ABC\equiv\angle DEF$. Segment addition
or subtraction gives, respectively in cases (1) and (2), $BC\equiv EF$, and finally
$\triangle ABC\equiv\triangle DEF$. We do not analyze $B-C-P$ explicitly, since it is
analogous to $P-B-C$ by exchanging $B$ with $C$.

If instead $\angle PAC\not\equiv \angle QDF$, let us analyze separately the two cases 
$B-P-C$ and $P-B-C$, case $P-C-B$ will follow by exchanging $B$ with $C$ in the second case.

If $B-P-C$, let us assume without loss of generality $\angle PAC >\angle QDF$ 
(see Figure \ref{PEN_LLL_LAL_f}). Then $\overrightarrow{AP}$ is inside $\angle BAC$, 
and by $\angle BAC\equiv\angle EDF$ and [\textbf{RCT}], $\overrightarrow{DQ}$ is inside 
$\angle EDF$, hence $E-Q-F$. By \ref{C4}, there exists a unique $P'$ with $E-P'-Q$ such 
that $\angle P'DF\equiv \angle PAC$. We then draw the perpendicular to $\overleftrightarrow{DP'}$ 
through $F$, with foot $I$, and the perpendicular through $E$, with foot $L$. We have $I\neq L$.
Indeed, if $I=L$, then $F$, $I$, $L$ and $E$ would be collinear on $\overleftrightarrow{EF}$ with $I=L=P'$. 
Therefore $\overleftrightarrow{DP'}$ is orthogonal to $\overleftrightarrow{EF}$. 
It follows that there are two perpendiculars to $\overleftrightarrow{EF}$ through $D$, 
contradicting Theorem \ref{PEN_HA}. Hence $I\neq L$.

Since $B-P-C$, the right angle $\angle APC$ is congruent to its supplement
$\angle APB$ with respect to $PC$; by Theorem \ref{PEN_HA_EA}, the angle opposite $PC$ in
$\triangle APC$, namely $\angle PAC$, is acute. Symmetrically, $\angle PAB$ is acute. By
\ref{RCT} (as before), $\angle P'DF\equiv\angle PAC$ and $\angle EDP'\equiv\angle PAB$, hence
both are acute as well (Proposition \ref{PEN_CP12}). Both feet $I$ and $L$ therefore lie on
the same ray from $D$ along $\overleftrightarrow{DP'}$.
Applying \ref{HA} gives $\triangle APC\equiv\triangle DIF$ and $\triangle APB\equiv
\triangle DLE$, whence $DI\equiv AP\equiv DL$. Since $I$ and $L$ lie on the same ray from
$D$, \ref{C1} forces $I=L$, contradicting $I\neq L$ established above. It follows that
$\angle PAC\not\equiv\angle QDF$ cannot hold, and therefore the assertion is proved.

If $P-B-C$, let us assume without loss of generality $\angle PAC <\angle QDF$ (see Figure \ref{PEN_LLL_LAL_f2}). 
Then $\angle QDF>\angle PAC>\angle BAC\equiv \angle EDF$, hence by angle ordering  $\overrightarrow{DE}$ is inside
$\angle QDF$, hence $Q-E-F$. By \ref{C4} we identify a unique ray $\overrightarrow{DP'}$ on the same side of $Q$ with respect 
to $\overleftrightarrow{DF}$ such that $\angle P'DF\equiv\angle PAC$. By the crossbar theorem we can directly assume $P'$  to be the intersection of $\overrightarrow{DP'}$ with $QF$, and from $\angle P'DF\equiv\angle PAC>\angle EDF$ it follows $P'-E-F$.  We then draw the perpendicular to $\overleftrightarrow{DP'}$
through $F$, with foot $I$, and the perpendicular through $E$, with foot $L$. We have $I\neq L$.
Indeed, if $I=L$, then $F$, $I$, $L$ and $E$ would be collinear on $\overleftrightarrow{EF}$ with $I=L=P'$.
Therefore $\overleftrightarrow{DP'}$ is orthogonal to $\overleftrightarrow{EF}$.
It follows that there are two perpendiculars to $\overleftrightarrow{EF}$ through $D$,
contradicting Theorem \ref{PEN_HA}. Hence $I\neq L$.

$\angle P'DE$ and $\angle P'DF$ are both acute angles being $<\angle QDF$ which is acute (Proposition \ref{PEN_CP12}). Both feet $I$ and $L$ therefore lie on
the same ray from $D$ along $\overleftrightarrow{DP'}$. Applying \ref{HA} gives $\triangle APC\equiv\triangle DIF$ and $\triangle APB\equiv
\triangle DLE$, whence $DI\equiv AP\equiv DL$. Since $I$ and $L$ lie on the same ray from
$D$, \ref{C1} forces $I=L$, contradicting $I\neq L$ established above. It follows that
$\angle PAC\not\equiv\angle QDF$ cannot hold, and therefore the assertion is proved.
\end{proof}
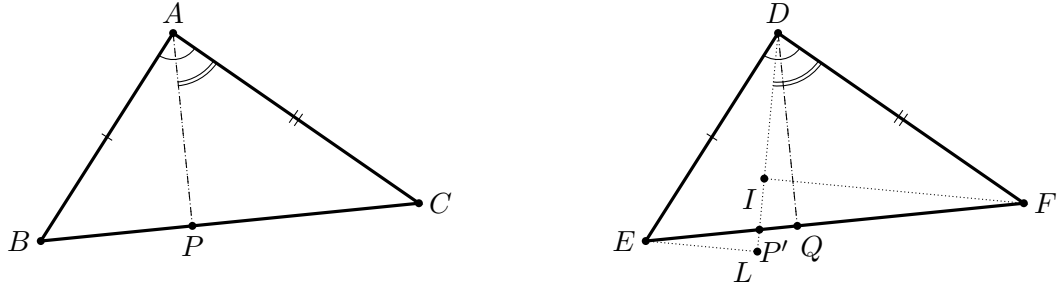
\begin{figure}
\centering
\begin{tikzpicture}
\coordinate (B) at (0,0);
\coordinate (C) at (5,0.5);
\coordinate (A) at (1.75,2.75);

\coordinate (E) at (8,0);
\coordinate (F) at (13,0.5);
\coordinate (D) at (9.75,2.75);

\draw[very thick] (A) -- (B) -- (C) -- cycle;
\draw[very thick] (D) -- (E) -- (F) -- cycle;

\fill (A) circle (1.5pt) node[above] {$A$};
\fill (B) circle (1.5pt) node[left] {$B$};
\fill (C) circle (1.5pt) node[right] {$C$};

\fill (D) circle (1.5pt) node[above] {$D$};
\fill (E) circle (1.5pt) node[left] {$E$};
\fill (F) circle (1.5pt) node[right] {$F$};
\coordinate (P) at ($(B)!(A)!(C)$);
\draw[dashed] (A) -- (P);

\fill (P) circle (1.5pt) node[below] {$P$};
\draw[densely dotted] (A) -- (P);
\coordinate (Q) at ($(E)!(D)!(F)$);
\draw[dashed] (D) -- (Q);

\fill (Q) circle (1.5pt) node[below] {$\quad Q$};
\draw[densely dotted] (D) -- (Q);
\coordinate (Pprime) at ($(E)!0.75!(Q)$);
\fill (Pprime) circle (1.5pt) node[below] {$\quad P'$};
\coordinate (I) at ($(D)!(F)!(Pprime)$);
\coordinate (L) at ($(D)!(E)!(Pprime)$);

\draw[densely dotted] (F) -- (I);
\draw[densely dotted] (E) -- (L);

\fill (I) circle (1.5pt) node[below] {$I\quad$};
\fill (L) circle (1.5pt) node[below] {$L\quad$};

\draw[densely dotted] (D) -- (L);

\PerpMark{A}{B}{0.5}{0.05}{thin}
\PerpMark{D}{E}{0.5}{0.05}{thin}
%
%
\PerpMark{A}{C}{0.49}{0.05}{thin}
\PerpMark{A}{C}{0.51}{0.05}{thin}
\PerpMark{D}{F}{0.49}{0.05}{thin}
\PerpMark{D}{F}{0.51}{0.05}{thin}
\pic[draw, angle radius=0.35cm] {angle=B--A--C};
\pic[draw, angle radius=0.35cm] {angle=E--D--F};

\pic[draw, angle radius=0.65cm] {angle=P--A--C};
\pic[draw, angle radius=0.70cm] {angle=P--A--C};

\pic[draw, angle radius=0.65cm] {angle=Pprime--D--F};
\pic[draw, angle radius=0.70cm] {angle=Pprime--D--F};
\end{tikzpicture}
\caption{Theorem \ref{PEN_LLL_LAL}-Proof schematic, case $B-P-C$}\label{PEN_LLL_LAL_f}
\end{figure}
\begin{figure}
\centering
\begin{tikzpicture}[scale=0.75]
\coordinate (B) at (0,0);
\coordinate (C) at (5,0.5);
\coordinate (A) at (-2.789,2.636);

\coordinate (E) at (8,0);
\coordinate (F) at (13,0.5);
\coordinate (D) at (5.211,2.636);

\draw[very thick] (A) -- (B) -- (C) -- cycle;
\draw[very thick] (D) -- (E) -- (F) -- cycle;

\fill (A) circle (1.5pt) node[above] {$A$};
\fill (B) circle (1.5pt) node[above right] {$B$};
\fill (C) circle (1.5pt) node[right] {$C$};

\fill (D) circle (1.5pt) node[above] {$D$};
\fill (E) circle (1.5pt) node[above right] {$E$};
\fill (F) circle (1.5pt) node[right] {$F$};
\coordinate (P) at ($(B)!(A)!(C)$);
\draw[dashed] (B) -- (P);

\fill (P) circle (1.5pt) node[below] {$P$};
\draw[densely dotted] (A) -- (P);
\coordinate (Q) at ($(E)!(D)!(F)$);
\draw[dashed] (E) -- (Q);

\fill (Q) circle (1.5pt) node[below] {$Q$};
\draw[densely dotted] (D) -- (Q);
\coordinate (Pprime) at ($(Q)!0.55!(E)$);
\fill (Pprime) circle (1.5pt) node[below] {$P'$};
\coordinate (I) at ($(D)!(F)!(Pprime)$);
\coordinate (L) at ($(D)!(E)!(Pprime)$);

\draw[densely dotted] (F) -- (I);
\draw[densely dotted] (E) -- (L);

\fill (I) circle (1.5pt) node[below] {$I$};
\fill (L) circle (1.5pt) node[above right] {$L$};

\draw[densely dotted] (D) -- (I);
\PerpMark{A}{B}{0.5}{0.05}{thin}
\PerpMark{D}{E}{0.5}{0.05}{thin}
\PerpMark{A}{C}{0.495}{0.025}{thin}
\PerpMark{A}{C}{0.505}{0.025}{thin}
\PerpMark{D}{F}{0.495}{0.025}{thin}
\PerpMark{D}{F}{0.505}{0.025}{thin}
%
%

\pic[draw, angle radius=0.35cm] {angle=B--A--C};
\pic[draw, angle radius=0.35cm] {angle=E--D--F};

\pic[draw, angle radius=0.65cm] {angle=P--A--C};
\pic[draw, angle radius=0.70cm] {angle=P--A--C};

\pic[draw, angle radius=0.65cm] {angle=Pprime--D--F};
\pic[draw, angle radius=0.70cm] {angle=Pprime--D--F};
\end{tikzpicture}
\caption{Theorem \ref{PEN_LLL_LAL}-Proof schematic, case $P-B-C$}\label{PEN_LLL_LAL_f2}
\end{figure}
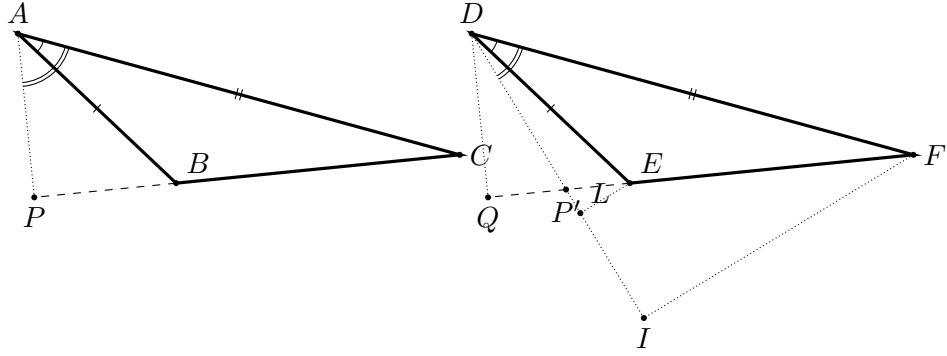
\section{Metamathematical Analysis}\label{PEN_s4}

Working within the axiomatic system $\MG^{-}$ introduced in Section~\ref{PEN_s2}, Theorem
\ref{PEN_ASA} yields
\begin{equation}\label{PEN_ALA}
\textrm{ASA},\; \textrm{[\textbf{RCT}]}
\vdash
\textrm{SAS}.
\end{equation}
Likewise, Theorem \ref{PEN_LLL_LAL} of Section~\ref{PEN_s3} yields
\begin{equation}\label{PEN_LLL}
\textrm{SSS},\;\textrm{[\textbf{RCT}]},\;\ref{MS},\;\ref{HA}\;\vdash\;\textrm{SAS}
\end{equation}
and
\begin{equation}\label{PEN_LAA}
\textrm{SAA},\;\textrm{[\textbf{RCT}]},\;\ref{AB},\;\ref{SA},\;[\mathbf{PA}]\;\vdash\;\textrm{SAS}.
\end{equation}
The following deductive equivalences follow immediately:
\begin{equation}\label{PEN_EALA}
\textrm{SAS}
\;\dashv\vdash\;
\textrm{ASA},\;\textrm{[\textbf{RCT}]},
\end{equation}
\begin{equation}\label{PEN_ESSS}
\textrm{SAS}
\;\dashv\vdash\;
\textrm{SSS},\;\textrm{[\textbf{RCT}]},\;\ref{MS},\;\ref{HA}
\end{equation}
and
\begin{equation}\label{PEN_ESAA}
\textrm{SAS}
\;\dashv\vdash\;
\textrm{SAA},\;\textrm{[\textbf{RCT}]},\;\ref{AB},\;\ref{SA},\;[\mathbf{PA}].
\end{equation}
The right-to-left direction of each equivalence relies on the proofs given above; the
left-to-right direction relies on the classical results already available in $\MG$ once SAS is
assumed \cite{Hilbert1950}.

These deductive equivalences offer a first comparison among the triangle congruence criteria
within $\MG^{-}$. To complement this analysis, we now turn to suitable models, which will let
us determine which principles and criteria can, or cannot, be derived from the remaining
axioms.

Following \cite{Hilbert1950}, the Cartesian plane $\mathbb{R}^{2}$ of elementary analytic
geometry is a model of $\MG$. Its points are the elements $(x,y)$ of $\mathbb{R}^{2}$, and its
lines are the sets defined by the parametric equations
\begin{equation}\label{PEN_s4_e1}
x=\lambda t+x_{0},\qquad y=\mu t+y_{0}.
\end{equation}
This choice of points and lines carries the full affine structure of the model and satisfies
the incidence axioms [\textbf{I}] and the order axioms [\textbf{O}]. Congruence between angles
is defined via angular measure, two angles being declared congruent exactly when their measures
coincide; congruence between segments is defined analogously via Euclidean length. Under these
definitions, axioms \ref{C1}--\ref{C5} and SAS/\ref{C6} all hold.

We denote this model, based on ordinary analytic geometry, by $\mathbb{E}^{2}$. We do not
verify here that it satisfies the various axioms of $\MG$ and $\MG^{-}$ -- an elementary
exercise in plane analytic geometry, carried out in detail by Hilbert himself
\cite[\S 9]{Hilbert1950} -- but simply record that, as a consequence, every result derived
from $\MG$ or $\MG^{-}$ holds in $\mathbb{E}^{2}$.
In particular, this model provides a consistency interpretation for all the principles examined
in the present analysis.

We now modify $\mathbb{E}^{2}$ by changing how segment length is defined, obtaining a new model
that we denote $\mathbb{E}_{H}^{2}$. This construction adapts to the plane the counterexample
that Hilbert used in \cite[\S 11]{Hilbert1950} to establish the independence of SAS from the
remaining congruence axioms.

Given two points $A_{1}(x_{1},y_{1})$ and $A_{2}(x_{2},y_{2})$ of the plane, define the length
of the segment $A_1A_2$ by
\begin{equation*}
L(A_{1}A_{2})=
\sqrt{(x_{1}-x_{2}+y_{1}-y_{2})^{2}+(y_{1}-y_{2})^{2}},
\end{equation*}
and declare two segments congruent when they share the same length in this sense. Segment
transport (\ref{C1}) and angle transport (\ref{C4}) still hold.

Segment addition (\ref{C3}) also holds: if a line is parametrized as in \eqref{PEN_s4_e1} and
$A_1,A_2,A_3$ correspond to parameters $t_1<t_2<t_3$, then
\begin{align*}
L(A_{1}A_{2})&=(t_{2}-t_{1})\sqrt{(\lambda+\mu)^{2}+\mu^{2}},\\
L(A_{2}A_{3})&=(t_{3}-t_{2})\sqrt{(\lambda+\mu)^{2}+\mu^{2}},\\
L(A_{1}A_{3})&=(t_{3}-t_{1})\sqrt{(\lambda+\mu)^{2}+\mu^{2}},
\end{align*}
and consequently
\begin{equation*}
L(A_{1}A_{3})=L(A_{1}A_{2})+L(A_{2}A_{3}).
\end{equation*}

We now show that SAS fails in $\mathbb{E}_{H}^{2}$. Consider the points
\begin{equation*}
O=(0,0),\qquad
A=(1,0),\qquad
B=(0,1),\qquad
C=\left(\frac12,\frac12\right),
\end{equation*}
shown in Figure~\ref{PEN_Hff}, and the triangles $\triangle OAC$ and $\triangle OBC$. Since $OC$
is shared by both, $OC\equiv OC$; moreover
\begin{align*}
L(AC)&=\sqrt{\left(\frac12-\frac12\right)^2+\left(-\frac12\right)^2}
=\frac12,\\
L(BC)&=\sqrt{\left(-\frac12+\frac12\right)^2+\left(\frac12\right)^2}
=\frac12,
\end{align*}
so $AC\equiv BC$. By elementary analytic geometry, the two angles at $C$ -- $\angle OCA$ and
$\angle BCO$ -- are both right angles, so the two triangles satisfy the SAS hypotheses.

\begin{figure}[ht!]
\centering
\begin{tikzpicture}[scale=4]

\coordinate (O) at (0,0);
\coordinate (A) at (1,0);
\coordinate (B) at (0,1);
\coordinate (C) at (.5,.5);

\draw[thick] (O)--(A)--(B)--cycle;
\draw[thick] (O)--(C);

\fill (O) circle (0.4pt);
\fill (A) circle (0.4pt);
\fill (B) circle (0.4pt);
\fill (C) circle (0.4pt);

\node[below left] at (O) {$O(0,0)$};
\node[below] at (A) {$A(1,0)$};
\node[left] at (B) {$B(0,1)$};
\node[right] at (C) {$C\left(\frac12,\frac12\right)$};

\end{tikzpicture}
\caption{Counterexample to SAS, ASA, and SAA.}
\label{PEN_Hff}
\end{figure}
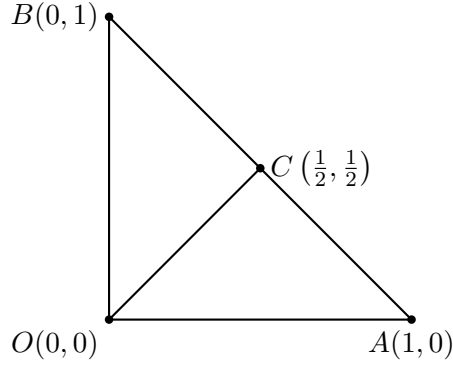

Yet the third sides tell a different story:
\begin{align*}
L(OA)&=\sqrt{(-1)^2}=1,\\
L(OB)&=\sqrt{(-1)^2+(-1)^2}=\sqrt2,
\end{align*}
so $OA\not\equiv OB$, and SAS fails in the model. By the soundness theorem
\cite[Sec.~2.5, pp.~131--134]{Enderton2001}, it follows that SAS is not provable in $\MG^{-}$.

The same triangles show that ASA and SAA fail as well. By elementary analytic geometry, the two
angles at $O$ are both $45^\circ$, so $\angle COB\equiv\angle COA$. Taking the shared side $CO$
together with its adjacent angles gives congruent angles on both sides, yet $OA\not\equiv OB$:
ASA fails. Taking instead $AC\equiv BC$ together with the pairs of angles $90^\circ$ and
$45^\circ$ gives the hypotheses of SAA, again with $OA\not\equiv OB$: SAA fails too. By the
soundness theorem once more, SAA is not provable in $\MG^{-}$.

We turn next to a counterexample for SSS. Consider the triangles $\triangle ABC$ and
$\triangle BCF$ (Figure~\ref{PEN_H_LLLf}), where
\begin{align*}
A&=(-2,-2),&
B&=(-2,-1),&
C&=(-1,-2),&
F&=(0,-2).
\end{align*}

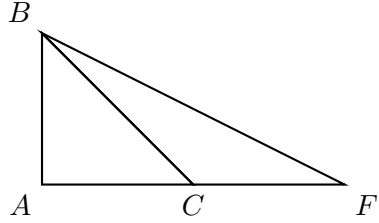
\begin{figure}[ht!]
\centering
\begin{tikzpicture}[scale=2.0,thick]

\coordinate[label=below left:$A$] (A) at (-2,-2);
\coordinate[label=above left:$B$] (B) at (-2,-1);
\coordinate[label=below:$C$] (C) at (-1,-2);
\coordinate[label=below right:$F$] (F) at (0,-2);

\draw (A)--(B)--(C)--cycle;
\draw (B)--(C)--(F)--cycle;

\end{tikzpicture}
\caption{Counterexample to SSS.}
\label{PEN_H_LLLf}
\end{figure}

Computing side lengths with $L$, for $\triangle ABC$ we find
\begin{align*}
L(AB)
&=
\sqrt{(0+1)^2+1^2}
=\sqrt2,\\
L(AC)
&=
\sqrt{(1+0)^2}
=1,\\
L(BC)
&=
\sqrt{(1-1)^2+(-1)^2}
=1,
\end{align*}
and for $\triangle BCF$,
\begin{align*}
L(BC)
&=1,\\
L(BF)
&=
\sqrt{(2-1)^2+(-1)^2}
=\sqrt2,\\
L(CF)
&=
\sqrt{(1+0)^2}
=1.
\end{align*}
Thus, up to a correspondence of vertices,
\begin{equation*}
AB\equiv BF,\qquad
BC\equiv BC,\qquad
AC\equiv CF,
\end{equation*}
and the SSS hypotheses are satisfied. The two triangles are nevertheless not congruent: the
angles of $\triangle ABC$ measure $45^\circ$, $45^\circ$, and $90^\circ$, while none of the
angles of $\triangle BCF$ is right. Since angle congruence here coincides with that of ordinary
Euclidean geometry,
\begin{equation*}
\triangle ABC\not\equiv\triangle BCF,
\end{equation*}
even though the SSS hypotheses hold. Hence SSS is not provable in $\MG^{-}$.

The same configuration also settles [\textbf{PA}]: the triangle $\triangle ABF$ is isosceles,
since $L(AB)=L(BF)=\sqrt2$ gives $AB\equiv BF$, yet $\angle BAF\not\equiv\angle BFA$, because
$\angle BAF$ is a right angle while $\angle BFA$ is not ($AF$ is not orthogonal to $BF$). Thus
[\textbf{PA}] fails in the model, and is therefore not provable in $\MG^{-}$.

Since angle congruence in $\mathbb{E}_{H}^{2}$ coincides with that of $\mathbb{E}^{2}$, every
purely angular result established earlier in this chapter continues to hold in
$\mathbb{E}_{H}^{2}$; in particular, [\textbf{RCT}] is readily seen to hold there; in particular, the entire ordering theory of Proposition \ref{PEN_CP12} transfers automatically, since its proof relies only on angle congruence. We thus
obtain, within $\MG^{-}$,
\begin{equation*}
\textrm{[\textbf{RCT}]}
\;\not\vdash\;
\textrm{SAS, SAA, SSS, [\textbf{PA}]}.
\end{equation*}
In other words, although [\textbf{RCT}] is compatible with $\MG^{-}$, it is on its own too weak
to reconstruct any of the triangle congruence criteria considered here. By contrast, we have
seen that adjoining ASA to [\textbf{RCT}] suffices to recover SAS, as recorded in the
equivalence \eqref{PEN_EALA}.

This equivalence does not, however, settle whether ASA alone -- without the help of
[\textbf{RCT}] -- already implies SAS; it therefore says nothing about the relative deductive
strength of the two criteria within the present theory. The situation can differ elsewhere: in
the metric framework adopted by \cite{Martin1982}, ASA and SAS are deductively equivalent outright,
since the proof of Theorem 16.5 there shows that ASA implies SAS, while the converse (SAS
implies ASA) already holds in that theory.

We conclude by examining the auxiliary principles \ref{MS}, \ref{AB}, [\textbf{SA}], and
\ref{HA}, which entered the reconstructions based on SSS and SAA. Since $\mathbb{E}^{2}$
satisfies every result of elementary analytic geometry, all of these principles are consistent
with $\MG$.

In the Hilbert model $\mathbb{E}_{H}^{2}$, \ref{MS}, \ref{AB}, and [\textbf{SA}] continue to
hold. For \ref{MS}: since the length of segments on a common line is proportional to the
difference of the corresponding parameters $t$, the midpoint corresponds simply to the average
of the endpoints' parameters. For \ref{AB}: angle measure coincides with that of elementary
analytic geometry, so every angle has a bisector. For [\textbf{SA}]: angle congruence again
coincides with the ordinary one, so supplements of congruent angles remain congruent.

A specific counterexample in $\mathbb{E}_{H}^{2}$ shows, however, that \ref{HA} fails there.
We now exhibit explicitly the counterexample to \ref{HA} announced above. Consider the
triangles $\triangle O_1P_1Q_1$ and $\triangle O_2P_2Q_2$ (Figure~\ref{PEN_H_HAf}), where
\begin{align*}
O_1&=(0,0), & P_1&=(1,0), & Q_1&=(0,1),\\
O_2&=(2,0), & P_2&=\left(\tfrac52,\tfrac12\right), & Q_2&=\left(\tfrac32,\tfrac12\right).
\end{align*}
By elementary analytic geometry, both triangles have a right angle at $O_1$, respectively
$O_2$, and both are isosceles right triangles: the acute angles at $P_1,Q_1$ and at $P_2,Q_2$
all equal $45^\circ$.

\begin{figure}[ht!]
\centering
\begin{tikzpicture}[scale=2.2,thick]

\coordinate[label=below:$O_1$] (O1) at (0,0);
\coordinate[label=below:$P_1$] (P1) at (1,0);
\coordinate[label=left:$Q_1$] (Q1) at (0,1);

\coordinate[label=below:$O_2$] (O2) at (2,0);
\coordinate[label=right:$P_2$] (P2) at (2.5,0.5);
\coordinate[label=left:$Q_2$] (Q2) at (1.5,0.5);

\draw (O1)--(P1)--(Q1)--cycle;
\draw (O2)--(P2)--(Q2)--cycle;

\fill (O1) circle (0.5pt);
\fill (P1) circle (0.5pt);
\fill (Q1) circle (0.5pt);
\fill (O2) circle (0.5pt);
\fill (P2) circle (0.5pt);
\fill (Q2) circle (0.5pt);

\end{tikzpicture}
\caption{Counterexample to HA: both triangles are Euclidean $45$-$45$-$90$ triangles with
$L(P_1Q_1)=L(P_2Q_2)=1$, yet their legs differ.}
\label{PEN_H_HAf}
\end{figure}
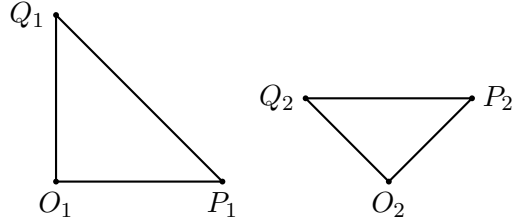

We compute the relevant lengths using $L$. For $\triangle O_1P_1Q_1$:
\begin{align*}
L(O_1P_1)&=\sqrt{(1)^2}=1,\\
L(O_1Q_1)&=\sqrt{(1)^2+(1)^2}=\sqrt2,\\
L(P_1Q_1)&=\sqrt{(-1+1)^2+1^2}=1.
\end{align*}
For $\triangle O_2P_2Q_2$:
\begin{align*}
L(O_2P_2)&=\sqrt{\left(\tfrac12+\tfrac12\right)^2+\left(\tfrac12\right)^2}
=\sqrt{\tfrac54}=\frac{\sqrt5}{2},\\
L(O_2Q_2)&=\sqrt{\left(-\tfrac12+\tfrac12\right)^2+\left(\tfrac12\right)^2}=\frac12,\\
L(P_2Q_2)&=\sqrt{(1)^2}=1.
\end{align*}

Both triangles therefore satisfy the hypotheses of \ref{HA}: a right angle
($\angle P_1O_1Q_1\equiv\angle P_2O_2Q_2$, both $90^\circ$), a congruent acute angle
($\angle O_1P_1Q_1\equiv\angle O_2P_2Q_2$, both $45^\circ$), and a congruent hypotenuse
($L(P_1Q_1)=L(P_2Q_2)=1$). Yet the legs are not congruent:
\begin{equation*}
L(O_1P_1)=1 \ \ne\ \frac{\sqrt5}{2}=L(O_2P_2), \qquad
L(O_1Q_1)=\sqrt2 \ \ne\ \frac12=L(O_2Q_2),
\end{equation*}
so $\triangle O_1P_1Q_1\not\equiv\triangle O_2P_2Q_2$. Hence \ref{HA} fails in
$\mathbb{E}_{H}^{2}$, and \ref{HA} is not provable in $\MG^{-}$.

Consequently, \ref{HA} -- together with [\textbf{PA}] and all the triangle congruence criteria
-- cannot be derived from \ref{MS}, \ref{AB}, [\textbf{RCT}], and [\textbf{SA}]. We thus
obtain, within $\MG^{-}$,
\begin{equation}\label{PEN_IND}
\textrm{\ref{MS}, \ref{AB}, [\textbf{SA}], [\textbf{RCT}]}
\;\not\vdash\;
\textrm{SAS, SAA, SSS, [\textbf{PA}], \ref{HA}}.
\end{equation}
\section{Conclusion}
For each of the congruence criteria ASA, SSS and SAA, we have identified a set of auxiliary principles that, together with the assumed criterion, allow us to deduce SAS.
The analysis of the models first makes it possible to distinguish the principles considered into two groups: those satisfied by the model $\mathbb{E}_{H}^{2}$,
\begin{equation*}
\textrm{[\textbf{RCT}]},\; \ref{MS},\; \ref{AB},\; \ref{SA},
\end{equation*}
and those that are not satisfied,
\begin{equation*}
\textrm{[\textbf{PA}]},\; \ref{HA}.
\end{equation*}
The latter are therefore independent of the axiomatic system consisting of $\MG^{-}$ and the principles satisfied by $\mathbb{E}_{H}^{2}$. Within both groups, a structural distinction can also be observed: some principles concern a single geometric individual, whereas others establish relations between distinct individuals.

Although we cannot rule out the existence in $\MG^{-}$ of the two synthetic proofs
\begin{equation*}
\textrm{SSS},\ [\textbf{RCT}],\; \ref{MS}\;\vdash\;\ref{HA},
\end{equation*}
and
\begin{equation*}
\textrm{SAA},\; [\textbf{RCT}],\ \ref{AB},\; [\textbf{SA}]\;\vdash\;\textrm{[PA]},
\end{equation*}
the different nature of the principles involved can provide a qualitative indication of their deductive strength. In the first case, SSS requires the principle \ref{HA}, which connects angles and segments belonging to distinct triangles; in the second, SAA requires the principle [\textbf{PA}], which describes a relation between angles and segments of the same triangle. This difference can be interpreted qualitatively as an indication of a lower deductive strength of SSS compared with SAA, while not allowing us to establish a formal relation of greater or lesser deductive power.

It is worth noting that this distinction in how \ref{HA} is reached does not translate
into two separate final arguments: once \ref{HA} is available -- whether assumed directly, as
in the SSS case, or derived via the congruence of all right angles (Theorem
\ref{PEN_EU4P}), as in the SAA case -- the remainder of the proof of Theorem
\ref{PEN_LLL_LAL} is identical for both criteria. In this sense, the reconstructions from
SSS and from SAA are not two independent developments but a single argument reached
through two different routes to the same intermediate result.

Among the three reconstructions, that of ASA nonetheless rests on a distinctly more secure
footing than the other two. Since $\mathbb{E}_{H}^{2}$ satisfies [\textbf{RCT}] while ASA fails
there, the soundness theorem shows definitively that
$\textrm{[\textbf{RCT}]}\not\vdash\textrm{ASA}$ in $\MG^{-}$: the reconstruction
\eqref{PEN_EALA} genuinely requires ASA, not merely [\textbf{RCT}] in disguise. No analogous
certainty is available for SSS or SAA. Since $\mathbb{E}_{H}^{2}$ satisfies neither SSS nor SAA,
it cannot address whether \ref{HA} is truly indispensable alongside SSS, [\textbf{RCT}], and
\ref{MS}, nor whether [\textbf{PA}] is truly indispensable alongside SAA, [\textbf{RCT}],
\ref{AB}, and [\textbf{SA}]: both questions, unlike the corresponding one for ASA, remain
entirely open, not merely unlikely to hold. In this precise sense, ASA occupies the highest
position among the three reconstructions considered here.

We mention, finally, that we also explored other candidate models in the course of this
investigation. Among these, a "moving protractor" model -- in which angle congruence is
redefined vertex by vertex through an arbitrary relabeling \cite{Brown2019}, while segment congruence remains
Euclidean -- and several variants of $\mathbb{E}_{H}^{2}$ obtained by replacing the matrix $M$
underlying the length distortion with other non-orthogonal matrices. In none of these attempts
did we find a model satisfying $\MG^{-}$ in which SAS fails while one of ASA, SSS, or SAA holds:
in every case examined, the failure of SAS was accompanied by the simultaneous failure of all
three remaining criteria. We report this only as a remark on the exploratory work underlying
the present analysis, not as a further result of the paper; a systematic study of this
phenomenon lies beyond its scope.
\section*{Axioms of $\MG$}
\begin{enumerate}[label=\textrm{[}\textbf{I}\arabic*\textrm{]}]
  \item\label{I1} \emph{For every point $P$ and for every point $Q\neq P$, there exists a unique line $l$ incident with $P$ and $Q$.}

  \item\label{I2} \emph{For every line $l$ there exist at least two distinct points that are incident with $l$.}

  \item\label{I3} \emph{There exist three distinct points with the property that no line is incident with all three of them.}
\end{enumerate}

\begin{enumerate}[label=\textrm{[}\textbf{O}\arabic*\textrm{]}]
  \item\label{O1} \emph{If $A-B-C$, then $A$, $B$, and $C$ are three distinct collinear points, and $C-B-A$.}

  \item\label{O2} \emph{Given any two distinct points $B$ and $D$, there exist points $A$, $C$, and $E$ lying on $BD$ such that
  $A-B-D$, $B-C-D$, and $B-D-E$.}

  \item\label{O3} \emph{If $A$, $B$, and $C$ are three distinct points lying on the same line, then one and only one of the points is between the other two.}

  \item\label{O4} \emph{For every line $l$ and for any three points $A$, $B$, and $C$ not lying on $l$:}
  \begin{enumerate}[label=(\roman*)]
    \item \emph{If $A$ and $B$ are on the same side of $l$ and $B$ and $C$ are on the same side of $l$, then $A$ and $C$ are on the same side of $l$.}

    \item \emph{If $A$ and $B$ are on opposite sides of $l$ and $B$ and $C$ are on opposite sides of $l$, then $A$ and $C$ are on the same side of $l$.}
  \end{enumerate}
\end{enumerate}

\begin{enumerate}[label=\textrm{[}\textbf{C}\arabic*\textrm{]}]
  \item\label{C1} \emph{If $A$ and $B$ are distinct points and if $A'$ is any point, then for each ray $r$ emanating from $A'$ there is a unique point $B'$ on $r$ such that $B'\neq A'$ and $AB\equiv A'B'$.}

  \item\label{C2} \emph{If $AB\equiv CD$ and $AB\equiv EF$, then $CD\equiv EF$. Moreover, every segment is congruent to itself.}

  \item\label{C3} \emph{If $A-B-C$, $A'-B'-C'$, $AB\equiv A'B'$, and $BC\equiv B'C'$, then $AC\equiv A'C'$.}

  \item\label{C4} \emph{Given an angle $\angle BAC$ (where, by definition of ``angle'', $AB$ is not opposite to $AC$), and given any ray $A'B'$ emanating from a point $A'$, then there is a unique ray $A'C'$ on a given side of line $A'B'$ such that
  $\angle B'A'C'\equiv\angle BAC$.}

  \item\label{C5} \emph{If $\angle A\equiv\angle B$ and $\angle A\equiv\angle C$, then $\angle B\equiv\angle C$. Moreover, every angle is congruent to itself.}

  \item\label{C6} \emph{If two sides and the included angle of one triangle are congruent respectively to two sides and the included angle of another triangle, then the two triangles are congruent.}
\end{enumerate}

\end{document}